\documentclass[12pt]{amsart}
\usepackage{latexsym, url}
\usepackage{amsthm}
\usepackage{amsmath}
\usepackage{amsfonts}
\usepackage{amssymb}
\usepackage[dvips]{graphicx}
\usepackage{xypic}
\input xy
\xyoption{all}

\newtheorem{theo}{Theorem}[section]
\newtheorem{lemma}[theo]{Lemma}

\newtheorem{propo}[theo]{Proposition}
\newtheorem{defi}[theo]{Definition}
\newtheorem{coro}[theo]{Corollary}
\newtheorem{rem}[theo]{Remark}

\newcommand\Inj{\operatorname{Inj}}
\newcommand\Sat{\operatorname{Sat}}

\newcommand\iso{\operatorname{iso}}

\newcommand\Mod{\operatorname{\bf Mod}}

\newcommand\id{\operatorname{id}}

\newcommand\ap{\operatorname{ap}}
\newcommand\Set{\operatorname{\bf Set}}
\newcommand\Met{\operatorname{\bf Met}}

\newcommand\Ban{\operatorname{\bf Ban}}

\newcommand\CMet{\operatorname{\bf CMet}}

\newcommand\cof{\operatorname{cof}}

\newcommand\ca{\mathcal {A}}

\newcommand\cu{\mathcal {U}}

\newcommand\ch{\mathcal {H}}
\newcommand\ci{\mathcal {I}}

\newcommand\ck{\mathcal {K}}
\newcommand\cl{\mathcal {L}}

\newcommand\cq{\mathcal {Q}}
\newcommand\eps{\varepsilon}
\newcommand\pa{\parallel}

 \newbox\noforkbox \newdimen\forklinewidth
\forklinewidth=0.3pt \setbox0\hbox{$\textstyle\smile$}
\setbox1\hbox to \wd0{\hfil\vrule width \forklinewidth depth-2pt
 height 10pt \hfil}
\wd1=0 cm \setbox\noforkbox\hbox{\lower 2pt\box1\lower
2pt\box0\relax}

\date{June 15, 2023}
 
\begin{document}
\title[Enriched purity and presentability in Banach spaces]
{Enriched purity and presentability in Banach spaces}
\author[J. Rosick\'{y}]
{J. Rosick\'{y}}
\thanks{Supported by the Grant Agency of the Czech Republic under the grant 22-02964S} 
\address{
\newline J. Rosick\'{y}\newline
Department of Mathematics and Statistics,\newline
Masaryk University, Faculty of Sciences,\newline
Kotl\'{a}\v{r}sk\'{a} 2, 611 37 Brno,\newline
Czech Republic}
\email{rosicky@math.muni.cz}

\begin{abstract}
The category $\Ban$ of Banach spaces and linear maps of norm 
$\leq 1$ is locally $\aleph_1$-presentable but not locally finitely presentable. We prove, however, that $\Ban$ is locally finitely presentable in the enriched sense over complete metric spaces. Moreover, in this sense, pure morphisms are just ideals of Banach spaces. We characterize classes of Banach spaces approximately injective with respect to sets of morphisms having finite-dimensional domains and separable codomains.
\end{abstract}
\maketitle
\section {Introduction}
A model theoretical approach to Banach spaces (and other structures from functional analysis) is well established (see, e.g., \cite{I}). Our aim
is to place this approach into the framework of the theory of locally presentable and accessible categories. The first steps in this direction were done in \cite{RT} and \cite{AR}.
Accessible categories were introduced by M. Makkai and R. Par\' e \cite{MP}
as a foundation of categorical model theory. Their theory was further developed in \cite{AR1} where, in particular, the importance of injectivity and purity have been understood. To deal with structures from functional analysis, we need this theory to be enriched over complete metric spaces. For instance, \cite{AR} showed that an analytical concept
of approximate injectivity coincides with enriched injectivity. \cite{AR}
also showed that finite-dimensional Banach spaces are finitely generated in
the enriched sense; here, morphisms of Banach spaces are linear maps of norm $\leq 1$. This provides a new perspective to the construction of a Gurarii space.

Our first result improves \cite{AR} by showing that finite-dimensional Banach spaces are even finitely presentable in the enriched sense. This implies that the category of Banach spaces is locally finitely presentable as an enriched category. Let us add that, as an ordinary category, it is only locally $\aleph_1$-presentable. We then prove that pure morphisms of Banach spaces in the enriched sense coincide with the well-established concept of ideals of Banach spaces (\cite{Ra}). We also show that pure
morphisms have the same model-theoretical meaning as pure morphisms in accessible categories when we use the logic of positive bounded formulas of \cite{H,I}. This complements the results of \cite{RT} where purity in metric enriched categories was introduced. Using purity, we characterize classes of Banach spaces approximately injective with respect to sets of morphisms having finite-dimensional domains and separable codomains.

{\bf Acknowledgement.} The author is extremaly grateful to the anonymous referee for very valuable comments and suggestions.
 
\section{Preliminaries}
We denote by $\CMet$ the category of complete metric spaces
and nonexpanding maps as morphisms where we allow distances to be $\infty$. 
This category is symmetric monoidal closed where the tensor product $X\otimes Y$ puts the $+$-metric  
$$
d((x,y),(x',y'))=d(x,x')+d(y,y')
$$ 
on the product $X\times Y$. The internal hom provides the hom-set  $\CMet(X,Y)$ with the sup-metric $d(f,g)=\sup\{d(fx,gx)\;|\;x\in X\}$. Moreover, $\CMet$ is locally $\aleph_1$-presentable (see \cite[2.3(2))]{AR}. A category $\ck$ is enriched over $\CMet$ if hom-sets $\ck(A,B)$ are complete metric spaces and the composition maps   $\ck(B,C)\otimes\ck(A,B)\to\ck(A,C)$ are nonexpanding. Our principal examples are the category $\CMet$ itself and the category $\Ban$ of Banach spaces and linear maps of norm $\leq 1$. A functor $F:\ck\to\cl$ between $\CMet$-enriched categories is \textit{enriched} if the mapping $\ck(A,B)\to\cl(FA,FB)$ is nonexpansive for all objects $A$ and $B$ in $\ck$. An adjunction $U\vdash F$ is \textit{enriched} if $\cl(FK,L)\cong\ck(K,UL)$ is an isomorphism of metric spaces for all objects $K$ in $\ck$ and $L$ in $\cl$. Consult \cite{AR1} for the concept of a locally presentable category and \cite{Ke} for the theory of enriched categories. Enriched categories over $\CMet$ were studied in \cite{RT} and \cite{AR} from where we recall the following concepts.

Given morphisms $f,g:A\to B$ in a $\CMet$-enriched category $\ck$ and $\eps\geq0$, we say that $f\sim_\eps g$ if $d(f,g)\leq\eps$ in the metric space $\ck(A,B)$. An $\varepsilon$-\textit{commutative} square 
$$
\xymatrix@=3pc{
A\ar[r]^{f} \ar[d]_{g} & B\ar[d]^{\bar{g}} \\
C \ar[r]_{\bar{f}} & D
}
$$
is a square such that $\bar{f}g\sim_\eps \bar{g}f$. 

\begin{defi}(\cite[2.2]{RT})
{
\em
Let $\eps\geq 0$. An $\eps$-commutative square
$$
\xymatrix@=3pc{
A\ar[r]^{f} \ar[d]_{g} & B\ar[d]^{\bar{g}} \\
C \ar[r]_{\bar{f}} & D
}
$$
is called an $\eps$-\textit{pushout} if for every $\varepsilon$-commutative square
$$
\xymatrix@=3pc{
A\ar[r]^{f} \ar[d]_{g} & B\ar[d]^{g'} \\
C \ar[r]_{f'} & D'
}
$$
there is a unique morphism $t:D\to D'$ such that $t\overline{f}=f'$ and $t\overline{g}=g'$.
}
\end{defi}
$0$-commutative squares are commutative squares and pushouts are $0$-pushouts.



\begin{rem}\label{e-push}
{
\em
An $\eps$-pushout in $\Ban$ is a square
$$
\xymatrix@=3pc{
A\ar[r]^{f} \ar[d]_{g} & B\ar[d]^{\overline{g}} \\
C \ar[r]_{\overline{f}} & D
}
$$
where $D$ is the coproduct $B\oplus C$ endowed with the norm
$$
\pa (x,y)\pa=\inf\{\pa b\pa + \pa c\pa + \eps\pa a\pa; x= b+f(a),y=c-g(a)\}
$$
(see \cite[6.5]{DR}). Hence an $\eps$-pushout of finite-dimensional (or separable) Banach spaces is finite-dimensional (separable resp.).
}
\end{rem}
 
\begin{rem}\label{iso}
{
\em
(1) A morphism $f:A\to B$ in a $\CMet$-enriched category $\ck$ is called an \textit{isometry} if for every $u,v:C\to A$ we have $d(u,v)= d(fu,fv)$. Isometries in $\CMet$ or $\Ban$ are the usual isometries.  

(2) Let $\ck$ have pushouts. We say that isometries in $\ck$ are \textit{stable under pushouts} if 
$$
\xymatrix@=3pc{
A\ar[r]^{f} \ar[d]_{g} & B\ar[d]^{\overline{g}} \\
C \ar[r]_{\overline{f}} & D
}
$$ 
is a pushout and $f$ an isometry then $\overline{f}$ is an isometry. Isometries in $\Ban$ are stable under pushouts (see \cite[A.19]{ASCGM}).

Following \cite[3.19]{AR}), isometries in $\Ban$ are also stable under $\eps$-pushouts.

(3) A morphism $f:A\to B$ in a $\CMet$-enriched category $\ck$ is called an $\eps$-\textit{isometry} provided that there are isometries $g:B\to C$ and $h:A\to C$ such that $gf\sim_{\eps} h$. If $\ck$ have $\eps$-pushouts for all $\eps\geq 0$ and isometries are stable under pushouts then, following \cite[3.22]{AR}, $f:A\to B$ is an $\eps$-isometry if and only if in the following $\eps$-pushout 
$$
\xymatrix@=3pc{
A\ar[r]^{f} \ar[d]_{\id_A} & B\ar[d]^{f_\eps} \\
A \ar[r]_{\overline{f}} & P
}
$$ 
$\overline{f}$ is an isometry. (The stability of isometries under pushouts is missing in \cite[3.22]{AR}).
}
\end{rem}

\begin{lemma}\label{e-iso1}
Let $\ck$ be a $\CMet$-enriched category with $\eps$-pushouts for all $\eps\geq 0$. Assume that isometries in $\ck$ are stable under pushouts. Then $f$ is an $\eps$-isometry if and only if there is an isometry $h$ and a morphism $g$ such that $gf\sim_\eps h$.
\end{lemma}
\begin{proof}
Necessity is evident and sufficiency follows from the proof of 
\cite[3.22]{AR}).
\end{proof}

\begin{lemma}\label{e-iso}
Let $\ck$ be a $\CMet$-enriched category with $\eps$-pushouts for all $\eps\geq 0$. Assume that isometries in $\ck$ are stable under pushouts. Let $f\sim_\eps g$ where $g$ is an isometry. Then $f$ is an $\eps$-isometry.
\end{lemma}
\begin{proof}
Take the $\eps$-pushout from \ref{iso}(3). Since $g\cdot\id_A\sim_\eps\id_B\cdot f$, there is $t:P\to B$ such that $t\overline{f}=g$. Since isometries are left cancellable, $\overline{f}$ is an isometry. Following \ref{iso}(3), $f$ is an $\eps$-isometry.
\end{proof}

\begin{lemma}\label{e-iso0}
Let $\ck$ be a $\CMet$-enriched category with $\eps$-pushouts for all $\eps\geq 0$. Assume that isometries in $\ck$ are stable under pushouts. If $f:A\to B$ is an $\eps$-isometry and $g:B\to C$ an $\delta$-isometry then $gf$ is an $(\eps+\delta)$-isometry.
\end{lemma} 
\begin{proof}
There are isometries $h_1:B\to D_1$, 
$h_2:A\to D_1$, $h_3:C\to D_2$ and $h_4:B\to D_2$ such that $h_1f\sim_\eps h_2$
and $h_3g\sim_\delta h_4$. Consider a pushout
$$
\xymatrix@=3pc{
B\ar[r]^{h_1} \ar[d]_{h_4} & D_1\ar[d]^{\overline{h}_4} \\
D_2 \ar[r]_{\overline{h}_1} & D
}
$$ 
Then
$$
\overline{h}_1h_3gf\sim_\delta \overline{h}_1h_4f=\overline{h}_4h_1f\sim_\eps \overline{h}_4h_2,
$$
and thus $\overline{h}_1h_3gf\sim_{\eps+\delta}\overline{h}_4h_2$. Since
$\overline{h}_1h_3$ and $\overline{h}_4h_2$ are isometries, $gf$ is
an $(\eps+\delta)$-isometry.
\end{proof} 

\begin{rem}\label{equivalent}
{
\em
(1) A linear mapping $f:A\to B$ between Banach spaces is called an 
$\eps$-\textit{isometry} if
$$
(1-\eps)\pa x\pa\leq\pa fx\pa\leq (1+\eps)\pa x\pa.
$$
This is equivalent to
$$
|\pa fx\pa-\pa x\pa|\leq\eps
$$
for every $x\in A$, $\pa x\pa\leq 1$ (see \cite[7.1(2)]{AR}).
$\eps$-isometries of norm $\leq 1$ are precisely $\eps$-isometries in the sense of \ref{iso}(3) (see \cite[7.2]{AR}). 

(2) There is a stronger concept of an $\eps$-isometry on Banach spaces, namely a linear mapping $f:A\to B$ such that 
$$
(1+\eps)^{-1}\pa x\pa\leq\pa fx\pa\leq (1+\eps)\pa x\pa.
$$
We will call this a \textit{strong $\eps$-isometry}. Since $(1-e)\leq (1+e)^{-1}$,
every strong $\eps$-isometry is an $\eps$-isometry.

(3) For $0<\eps<1$, an $\eps$-isometry is a strong 
$\frac{\eps}{1-\eps}$-isometry. Indeed, $(1+\frac{\eps}{1-\eps})^{-1}=1-\eps$ and $\eps\leq\frac{\eps}{1-\eps}$.

(4) Every $\eps$-isometry is injective. Indeed, if $\pa fx\pa=0$ then
$\pa x\pa=0$.
}
\end{rem}

For a $\CMet$-enriched category $\ck$, the arrow category $\ck^\to$ is a $\CMet$-enriched category with the hom-space $\ck^\to(f,g)$ of $f:A\to B$ and $g:C\to D$ defined by the following pullback in $\Met$
$$
		\xymatrix@=3pc{
			\ck^\to(f,g) \ar [r]^{}\ar[d]_{} & \ck(A,C) \ar[d]^{\ck(A,g)} \\
			 \ck(B,D)  \ar [r]_{\ck(f,D)}& \ck(A,D)
		}
		$$ 
Recall that objects of $\ck^\to$ are morphisms $f:A\to B$ of $\ck$ and morphisms $(u,v):f\to g$ are	given by commutative squares
$$
		\xymatrix@=3pc{
			A \ar [r]^{f}\ar[d]_{u} & B \ar[d]^{v}\\
			 C  \ar [r]_{g}& D
		}
		$$ 
 
\section{Finite presentability}
In an ordinary category $\ck$, an object $A$ is finitely presentable if its hom-functor $\ck(A,-):\ck\to\Set$ preserves directed colimits. If $\ck$ is enriched over $\CMet$, an object $A$ is finitely presentable (in the enriched sense) if its hom-functor $\ck(A,-):\ck\to\CMet$ preserves directed colimits. While only the trivial Banach space $\{0\}$ is finitely presentable in the ordinary sense (see \cite[2.7(1)]{AR}), we prove that finite-dimensional Banach spaces are finitely presentable (in the enriched sense). This improves \cite[7.6]{AR} which showed that hom-functors $\ck(A,-):\Ban\to\CMet$ with $A$ finite-dimensional preserve directed colimits of isometries.

\begin{theo}\label{finpres}
Finite-dimensional Banach spaces are finitely presentable in the enriched sense.
\end{theo}
\begin{proof}
Let $A$ be finite-dimensional and $(k_{ij}:K_i\to K_j)_{i\leq j\in I}$ be a directed diagram in $\Ban$ with a colimit $(k_i:K_i\to K)_{i\in I}$. 
Following \cite[2.5]{AR}, we have to prove the following
\begin{enumerate}
\item Given $f:A\to K$ and $\eps>0$ then there are $i\in I$ and $g:A\to K_i$  
 such that $k_ig\sim_\eps f$.
\item Given $f,g:A\to K_i$ then 
$$
\pa k_if-k_ig\pa=\inf_{j\geq i}\pa k_{ij}f-k_{ij}g\pa.
$$
\end{enumerate}
We can assume that $I$ is a well-ordered chain (see \cite[1.7]{AR1}). If the cofinality $\cof(I)$ of $I$ is $>\aleph_0$ then $I$ is $\aleph_1$-directed and the result follows from the $\aleph_1$-presentability of finite-dimensional Banach spaces (\cite[2.7(2)]{AR}). Assume that $\cof(I)=\aleph_0$ and let $i_0,i_1\dots,i_t,\dots$ be a cofinal chain in $I$. 

The proof of (1): Every morphism $f:A\to K$ has a factorization $f=f_2f_1$ where $f_2:B\to K$ is an isometry and $f_1:A\to B$ is surjective. Thus $B$ is finite-dimensional provided that $A$ is finite-dimensional and therefore,
we can assume, without a loss of generality, that $f$ is an isometry. 
At first, we prove
 
($\ast$) For every $\varepsilon>0$ there exist $i\in I$ and an $\varepsilon$-isometry $f^\ast\colon A\to K_i$ (not necessarily of norm $\leq 1$) such that $f\sim_\varepsilon k_if^\ast$.  
 
Let $e_1,\dots,e_n$ be a basis of $A$. Since any two norms on a finite-dimensional Banach space are equivalent, there is a number $r$ such that 
$$
\sum_{0<m\leq n} |a_m|\leq r\parallel\sum_{0<m\leq n} a_me_m\parallel
$$
for all $a_1,\dots,a_n\in\Bbb C$.
Let $\delta = \frac{\varepsilon}{2r}$. Following 
\cite[2.5(4)]{AR}, there are elements 
$u_1,\dots,u_n\in K_i$, $i<\mu$, such that 
$$
\parallel k_iu_m-fe_m\parallel\leq\delta
$$
for $m=1,\dots, n$. We can assume that $i=i_0$. Let $f':A\to K_i$ be the linear mapping such that $f'e_m=u_m$ for $m=1,\dots,n$. 
We have
\begin{align*}
\parallel(k_if'-f)(\sum_{0<m\leq n} a_me_m)\parallel &\leq\sum_{0< m\leq n} |a_m|\parallel(k_if'-f)(e_j)\parallel\leq\sum_{0< m\leq n}|a_m|\delta\\
&\leq\frac{\varepsilon}{2}\parallel\sum_{0<m\leq n} a_me_m\parallel.
\end{align*}
Hence $k_if'\sim_\frac{\varepsilon}{2} f$.

Since $\pa k_{ii_t}f'\pa$ converges pointwise to $\pa k_if'\pa$ (see \cite[2.5(4)]{AR}) and the closed unit ball $A_1$ of $A$ is compact, the convergence is uniform on $A_1$ following Dini's theorem (see \cite[p. 165]{E}). Hence
$$
\pa k_{ii_t}f'\pa -\pa k_if'\pa\leq \frac{\eps}{2}
$$ 
for some $t$. We will show that $f^\ast=k_{ii_t}f':A\to K_{i_t}$ is an $\eps$-isometry. Since $f^\ast$ does not need to have norm $\leq 1$, we
cannot use \ref{e-iso1} but we will apply \ref{equivalent}(1).

Let $a=\sum\limits_{0<m\leq n} a_me_m\in A_1$ and $\pa a\pa\leq 1$. We have
\begin{align*}
|\pa f^\ast a\pa -\pa a\pa|
&=|\pa k_{ii_t}f'a\pa -\pa k_if'a\pa +\pa k_if'a\pa -\pa a\pa|\\
&\leq |\pa k_{ii_t}f'a\pa -\pa k_if'a\pa| +|\pa k_if'a\pa -\pa a\pa|\\
&\leq \frac{\eps}{2}+|\pa k_if'a\pa -\pa fa\pa|\leq\frac{\eps}{2}+\frac{\eps}{2}\pa a\pa\leq\eps.
\end{align*}
Hence $f^\ast$ is an $\eps$-isometry and 
$$
f\sim_{\frac{\eps}{2}} k_if'=k_{i_t}k_{ii_t}f'=k_{i_t}f^\ast.
$$ 
This proves ($\ast$).

Now, we prove (1). We take $f^\ast$ from ($\ast$) for $\varepsilon'= \frac{\varepsilon}{2}$. Let $\parallel\sum_{0<j\leq n}a_je_j\parallel=1$. 
Since $f^\ast$ is an $\eps'$-isometry, following \ref{equivalent}(1), we have
$$
|\parallel f^\ast(\sum_{0<j\leq n}a_je_j)\parallel -1|\leq\varepsilon'.
$$
If $\parallel f^\ast(\sum_{0<j\leq n}a_je_j)\parallel\geq 1$ then $\parallel f^\ast(\sum_{0<j\leq n}a_je_j)\parallel\leq 1+\varepsilon'$.
If $\parallel f^\ast(\sum_{0<j\leq n}a_je_j)\parallel \leq 1$ then, again, $\parallel f^\ast(\sum_{0<j\leq n}a_je_j)\parallel\leq 1+\varepsilon'$. We have proved that $\parallel f^\ast\parallel\leq 1+\varepsilon'$.
Hence $g= \frac{1}{1+\varepsilon'} f^\ast$ is a morphism in $\Ban$.

For $a=\sum_{0<j\leq n}a_je_j$ we have
$$
\parallel  f^\ast a - ga\parallel = \frac{1}{1+\varepsilon'} \parallel (1+\varepsilon')f^\ast a - f^\ast a\parallel = \frac{\varepsilon'}{1+\varepsilon'}\parallel f^\ast a\parallel\leq\frac{\varepsilon'}{1+\varepsilon'}(1+\varepsilon')\parallel a\parallel=\varepsilon'\parallel a\parallel.
$$
Hence $\parallel f^\ast -g\parallel\leq\varepsilon'$ and thus $\parallel k_{i_t}f^\ast -k_{i_t}g\parallel\leq\varepsilon'$. Since
$\parallel k_{i_t}f^\ast - f\parallel\leq\varepsilon'$, we have $\parallel k_{i_t}g - f\parallel\leq\varepsilon$.
 
The proof of (2): We have to show that
$$
\pa k_if-k_ig\pa=\inf_{j\geq i}\pa k_{ij}f-k_{ij}g\pa
$$
for $f,g:A\to K_i$. This means that   
$$
\pa k_if-k_ig\pa=\lim_t\pa k_{ii_t}f-k_{ii_t}g\pa.
$$
Following \cite[2.5(4)]{AR}, $\pa k_{ii_t}(f-g)\pa$ converges pointwise
to $\pa k_i(f-g)\pa$ on the closed unit ball $A_1$ of $A$. Using Dini's theorem, we get that the convergence is uniform. Hence 
$$
\pa k_i(f-g)\pa =\lim_t\pa k_{ii_t}(f-g)\pa .
$$
\end{proof}

Locally finitely presentable enriched categories were introduced in \cite{Ke1}. Since $\CMet$ is not locally finitely presentable as an ordinary category, \cite{Ke1} does not fully apply to $\CMet$-enriched categories. 

We say that a $\CMet$-enriched category is \textit{locally finitely presentable} if it has all weighted colimits and a set $\ca$ of finitely presentable objects (in the enriched sense) such that every object is a directed colimit of objects of $\ca$. 

These categories are locally finitely presentable in the sense of \cite{Ke1} because they have a strong generator consisting of finitely presentable objects. However, $\CMet$ is locally finitely presentable
in the sense of \cite{Ke1} (see (3.4) in this paper) but not locally finitely presentable in our sense.

\begin{coro}\label{finpres1}
$\Ban$ is locally finitely presentable as a category enriched over $\CMet$.
\end{coro}
\begin{proof}
$\Ban$ has all weighted colimits (see \cite[4.5(3)]{AR}. Thus the result
follows from \ref{finpres} and the fact that every Banach space is a directed colimit of finite-dimensional spaces.
\end{proof}

\begin{rem}\label{finpres2}
{
\em
\cite[7.7]{AR} characterizes Banach spaces $A$ such that $\Ban(A,-):\Ban\to\CMet$ preserves directed colimits of isometries
as those admitting for every $\eps>0$ a morphism $u:A\to A_0$ to a finite-dimensional Banach space with $r:A_0\to A$ such that $ru\sim_\eps\id_A$. These Banach spaces do not need to be finite-dimensional and satisfy (1) above for every directed colimit $(k_i:K_i\to K)$. Indeed, for $f:A\to K$ and $\eps>0$, we take $u:A\to A_0$
and $r:A_0\to A$ with $A_0$ finite-dimensional and $ru\sim_{\frac{\eps}{2}}\id_A$.
Following \ref{finpres}, there is $g:A_0\to K_i$ such that $k_ig\sim_{\frac{\eps}{2}} fr$. Then
$$
k_igu\sim_{\frac{\eps}{2}}fru\sim_{\frac{\eps}{2}} f.
$$ 
Hence $k_igu\sim_\eps f$.

We do not know whether these Banach spaces $A$ can satisfy (2) as well.
So, we do not know whether finitely presentable Banach spaces are finite-dimensional.
}
\end{rem}

\begin{propo}\label{adj1}
Let $U:\ck\to\cl$ and $F:\cl\to\ck$ be an enriched adjunction (where $F$ is a left adjoint) such that $U$ preserves directed colimits. Then $F$ preserves finitely presentable objects (in the enriched sense).
\end{propo}
\begin{proof}
Let $A$ be finitely presentable in $\cl$. Since
$$
\ck(FA,-)\cong\cl(A,U-)
$$
and $\cl(A,U-)$ preserves directed colimits, $FA$ is finitely presentable.
\end{proof}

\section{Purity}
A morphism $f:K\to L$ in an ordinary locally finitely presentable category $\ck$ is pure if in every commutative square
$$
\xymatrix@=3pc{
K \ar[r]^{f} & L \\
A\ar [u]^{u} \ar [r]_{g} & B \ar[u]_{v}
}
$$
with $A$ and $B$ finitely presentable there exists $t:B\to K$ such that $tg=u$. Pure morphisms are precisely directed colimits of split monomorphisms in $\ck^\to$ (see \cite[2.30]{AR1}). Recall that a monomorphism $f:K\to L$ is split if there exists $s:L\to K$ such that $sf=\id_K$.  

Since $\Ban$ is not locally finitely presentable as an ordinary category, one cannot expect that pure morphisms can be defined in this way.  

\begin{defi}\label{pure}
{
\em
A morphism $f:K\to L$ in $\Ban$ is  \textit{pure} if for every  $\eps>0$ and every commutative square 
$$
\xymatrix@=3pc{
K \ar[r]^{f} & L \\
A\ar [u]^{u} \ar [r]_{g} & B \ar[u]_{v}
}
$$
with $A$ and $B$ finite-dimensional, there exists $t:B\to K$ such that   $tg\sim_\eps u$.  
}
\end{defi}  

In \cite[5.2]{RT}, these morphisms were called barely ap-pure.

\begin{lemma}\label{pure2}
A morphism $f:K\to L$ in $\Ban$ is pure if and only if for every $\eps>0$
and every $\eps$-commutative square
$$
\xymatrix@=3pc{
K \ar[r]^{f} & L \\
A\ar [u]^{u} \ar [r]_{g} & B \ar[u]_{v}
}
$$
with $A$ and $B$ finite-dimensional, there exists $t:B\to K$ such that $tg\sim_{2\varepsilon} u$. 
\end{lemma} 
\begin{proof}
Sufficiency is evident. Let $f$ be pure and consider an 
$\eps$-commutative square
$$
\xymatrix@=3pc{
K \ar[r]^{f} & L \\
A\ar [u]^{u} \ar [r]_{g} & B \ar[u]_{v}
}
$$
with $A$ and $B$ finite-dimensional. Consider an $\eps$-pushout
$$
\xymatrix@=3pc{
A \ar[r]^{\overline{g}} & C \\
A\ar [u]^{\id_A} \ar [r]_{g} & B \ar[u]_{g_{\eps}}
}
$$
There is a unique morphism $t:C\to L$ such that $t\overline{g}=fu$ and
$tg_\eps=v$. Thus we get a commutative square
$$
\xymatrix@=3pc{
K \ar[r]^{f} & L \\
A\ar [u]^{u} \ar [r]_{\overline{g}} & C \ar[u]_{t}
}
$$
Following \ref{e-push}, $C$ is finite-dimensional and thus there exists $w:C\to K$ such that $w\overline{g}\sim_\eps u$. Hence
$$
wg_\eps g\sim_\eps w\overline{g}\sim_\eps u.
$$
Hence $wg_\eps g\sim_{2\eps } u$.  
\end{proof}

\begin{rem}\label{weakly}
{
\em
In the terminology of \cite{RT}, this means that barely ap-pure and weakly ap-pure mophisms coincide. The same holds for barely $\lambda$-ap-pure and weakly $\lambda$-ap-pure morphisms.
}
\end{rem}

\begin{lemma}\label{pure3}
A morphism $f:K\to L$ in $\Ban$ is  \textit{pure} if and only if for every  $\eps>0$ and every commutative square 
$$
\xymatrix@=3pc{
K \ar[r]^{f} & L \\
A\ar [u]^{u} \ar [r]_{g} & B \ar[u]_{v}
}
$$
with $A$ and $B$ finitely presentable in the enriched sense, there exists $t:B\to K$ such that   $tg\sim_\eps u$.  
\end{lemma}
\begin{proof}
Sufficiency follows from \ref{finpres}. Assume that $f$ is pure and consider a  commutative square 
$$
\xymatrix@=3pc{
K \ar[r]^{f} & L \\
A\ar [u]^{u} \ar [r]_{g} & B \ar[u]_{v}
}
$$
with $A$ and $B$ finitely presentable in the enriched sense. Let $\eps>0$
and take $u_A:A\to A_0$, $r_A:A_0\to A$, $u_B:B\to B_0$ and $r_B:B_0\to B$ from \ref{finpres2} such that $r_Au_A\sim_{\frac{\eps}{4}}\id_A$ and $r_Bu_B\sim_{\frac{\eps}{4}}\id_B$. Then the square
$$
\xymatrix@=3pc{
K \ar[r]^{f} & L \\
A_0\ar [u]^{ur_A} \ar [r]_{u_Bgr_A} & B_0 \ar[u]_{vr_B}
}
$$
$\frac{\eps}{4}$-commutes because
$$
vr_Bu_Bgr_A\sim_{\frac{\eps}{4}}vgr_A=fuu_A. 
$$
Following \ref{pure2}, there exists $t:B_0\to K$ such that $tu_Bgr_A\sim_{\frac{\eps}{2}}ur_A$. Then
$$
tu_Bg\sim_{\frac{\eps}{4}}tu_Bgr_Au_A\sim_{\frac{\eps}{2}}ur_Au_A
\sim_{\frac{\eps}{2}}u
$$
Hence $tu_Bg\sim_\eps u$.
\end{proof}

\begin{rem}\label{epure}
{
\em
Ordinary purity can be expressed in the following way. Let $\cq(g,f)$ be the image of the projection $p_1:\ck^\to(g,f)\to \ck(A,K)$ sending $(u,v)$ to $u$. Hence $\cq(g,f)$ consists of those $u:A\to K$ for which there exists $v:B\to L$ such that $vg=fu$. Then $f$ is pure if and only if the mapping 
$$
q:\ck(B,K)\to\cq(g,f)
$$
sending $t:B\to K$ to $tg$ is surjective for every $g:A\to B$ with $A$ and $B$ finitely presentable.

The approach of \cite{LaR} suggests how to define enriched purity. We replace the factorization system (surjective, injective) in $\Set$ by
the factorization system (dense, isometry) in $\CMet$ (see 
\cite[3.16(2)]{AR}). Then $\cq(g,f)$ is given by the (dense, isometry) factorization of $p_1$ and we demand that $q$ above is dense. Using 
\ref{pure3}, this leads to the just defined pure morphisms.
}
\end{rem}

\begin{propo}\label{closure}
Pure morphisms are closed under directed colimits of in $\Ban^\to$.
\end{propo}
\begin{proof}
Let $(k_i,l_i):(f_i\to f)_{i\in I}$ be a directed colimit 
of pure morphisms $f_i:K_i\to L_i$ in $\Ban^\to$. This means that $k_i:K_i\to K$ and $l_i:L_i\to L$ are directed colimits in $\Ban$ and $fk_i=l_if_i$ for every $i\in I$. Consider a commutative square 
$$
\xymatrix@=3pc{
K \ar[r]^{f} & L \\
A\ar [u]^{u} \ar [r]_{g} & B \ar[u]_{v}
}
$$
where $A$ and $B$ are finite-dimensional and take $\eps>0$. Following 
\ref{finpres}(1), there is $i\in I$ and $u':A\to K_i$ and $v':B\to L_i$ such that $k_iu'\sim_{\frac{\eps}{3}} u$ and $l_iv'\sim_{\frac{\eps}{3}} v$. Since
$$
l_if_iu'=fk_iu'=fu=vg=l_iv'g,
$$
\ref{finpres}(2) implies that there is $j\geq i$ such that
$f_ju'\sim_{\frac{\eps}{3}} v'g$. Following \ref{pure2}, there is 
$t:B\to K_j$ such that $tg\sim_{\frac{2\eps}{3}} u'$. Hence 
$$
k_jtg\sim_{\frac{2\eps}{3}} k_ju'\sim_{\frac{\eps}{3}} u.
$$
Thus $k_jtg\sim_\eps u$.
\end{proof}
\begin{rem}\label{split}
{
\em
(1) Split monomorphisms $f:K\to L$ in $\Ban$ coincide with projections 
$p:L\to L$ of norm one on $L$. Indeed, if $f$ is split by $s$ (i.e., $sf=\id_K$) then $p=fs$. Conversely, given a projection $p:L\to L$ of norm one then the embedding $p[L]\to L$ is split by $p$ (see \cite[2.2.9]{BS}). Equivalently, $L$ is a biproduct of $p[L]$ and $\ker(p)$.

(2) Every split monomorphism in $\Ban$ is pure. It suffices to take $t=vs$.
}
\end{rem}

\begin{propo}\label{char}
A morphism $f$ in $\Ban$ is pure if and only if for every $\eps>0$ there are $h_1,h_2$ such that $h_1f\sim_\eps h_2$ and $h_2$ is a directed colimit of split monomorphisms in $\Ban^\to$. 
\end{propo}
\begin{proof}
Assume that, for every $\eps>0$, there are $h_1,h_2$ such that 
$h_1f\sim_\eps h_2$ where $h_2$ is a directed colimit of split monomorphisms in $\Ban^\to$. Consider a commutative square
$$
\xymatrix@=3pc{
K \ar[r]^{f} & L \\
A\ar [u]^{u} \ar [r]_{g} & B \ar[u]_{v}
}
$$
where $A$ and $B$ are finite-dimensional and take $\eps>0$ and $h_1:L\to M,h_2:K\to M$ above for $\frac{\eps}{2}$. Then the square
$$
\xymatrix@=3pc{
K \ar[r]^{h_2} & M \\
A\ar [u]^{u} \ar [r]_{g} & B \ar[u]_{h_1v}
}
$$
$\frac{\eps}{2}$-commutes. Since split monomorphisms are pure, \ref{closure} implies that $h_2$ is pure. Following \ref{pure2}, there is $t:B\to K$ such that  $tg\sim_\eps u$. Thus $f$ is pure.

Conversely,  let $f:K\to L$ be pure and $\eps>0$. Express $K$ as a directed colimit $(k_i:K_i\to K)_{i\in I}$ of isometries $k_{ii'}:K_i\to K_{i'}$ between finite-dimensional Banach spaces. Similarly, let $(l_j:L_j\to L)_{j\in J}$ be a directed colimit of isometries between finite-dimensional Banach spaces. Following \ref{finpres}
(or \cite[7.6]{AR}), for every $i\in I$ there is $f_i:K_i\to L_{j_i}$ such that $l_{j_i}f_i\sim_{\frac{\eps}{2}}fk_i$.
Following \ref{pure2}, for every $i\in I$ there is $t_i:L_{j_i}\to K$ such that $t_if_i\sim_\eps k_i$. Like in the proof of
\cite[3.11]{RT}, in the $\eps$-pushouts
$$
\xymatrix@=3pc{
K \ar[r]^{\bar{f}_i} & \bar{L}_i \\
K_i\ar [u]^{k_i} \ar [r]_{f_i} & L_{j_i} \ar[u]_{\bar{k}_i}
}
$$ 
every $\bar{f}_i$ is a split monomorphism. Since
$$
\bar{f}_{i'}k_i=\bar{f}_{i'}k_{i'}k_{ii'}=\bar{k}_{i'}f_{i'}k_{ii'}=
\bar{k}_{i'}l_{j_ij_{i'}}f_i,
$$
there is a unique morphism $\bar{k}_{ii'}:\bar{L}_i\to\bar{L}_{i'}$
such that $\bar{k}_{ii'}\bar{f}_i=\bar{f}_{i'}$ and $\bar{k}_{ii'}\bar{k}_i=\bar{k}_{i'}l_{j_ij_{i'}}$.

We get a directed diagram
$(\id_K,\bar{k}_{ii'}):\bar{f}_i\to\bar{f}_{i'}$ of split monomorphisms
whose colimit $\bar{f}:K\to\bar{L}$ is given by the $\eps$-pushout
$$
\xymatrix@=3pc{
K \ar[r]^{h_2} & \bar{L} \\
K\ar [u]^{\id_K} \ar [r]_{f} & L \ar[u]_{h_1}
}
$$ 
\end{proof}

\begin{rem}\label{char1}
{
\em
In the proof above,  $\bar{k}_{ii'}$ are isometries (see \ref{iso}(2)). Hence $h_2$ can be taken as a directed colimit of split monomorphisms and isometries in $\Ban^\to$.
}
\end{rem}

\begin{lemma}\label{cancell}
Pure morphisms in $\Ban$ are left-cancellable, i.e., $f_2f_1$ pure implies that $f_1$ is pure.
\end{lemma}
\begin{proof}
It follows from \ref{char}.
\end{proof}

\begin{lemma}
Pure morphisms in $\Ban$ are closed under composition.
\end{lemma}
\begin{proof}
Let $f_1:K\to L$ and $f_2:L\to M$ be pure. Consider a commutative diagram
$$
\xymatrix@=3pc{
K \ar[r]^{f_1} & L \ar[r]^{f_2}  & M\\
A\ar [u]^{u} \ar [r]_{g} & B\ar[ru]_{v}
}
$$ 
and $\eps>0$. Since $f_2$ is pure, there is $t:B\to L$ such that $tg\sim_{\frac{\eps}{2}} f_1u$. Following \ref{pure2}, there is $s:B\to K$ such that $sg\sim_{\eps} u$. Thus $f_2f_1$ is pure.
\end{proof}
 
\begin{lemma}\label{pure6}
Pure morphisms in $\Ban$ are isometries.
\end{lemma}
\begin{proof}
Since isometries are closed under directed colimits (\cite[2.5(4)]{AR}) and split mo\-no\-mor\-phisms are isometries, $h_2$ in \ref{char} are isometries. Following \ref{char} and \ref{e-iso1}, a pure morphism $f$ is an $\eps$-isometry for every $\eps>0$. Hence $f$ is an isometry (see
\ref{equivalent}(1)).
\end{proof}

Recall that a closed subspace $K$ of a Banach space $L$ is called 
an \textit{ideal} if the anihilator of $K$ in the dual space $L^\ast$ is the kernel of norm one projection in $L^\ast$. Ideals $K$ in $L$ are characterized as closed subspaces $K$ such that for every finite-dimensional subspace $B$ of $L$ and every $\eps>0$, there exists a linear map $t:B\to K$ such that $tx=x$ for $x\in K\cap B$ and $\pa t\pa\leq 1+\eps$ (see \cite[Theorem 1]{Li}).

\begin{propo}\label{ideal}
A closed subspace $K$ of a Banach space $L$ is an ideal if and only if
the embedding $K\to L$ is pure.
\end{propo}
\begin{proof}
I. Let $f:K\to L$ be an embedding of an ideal. Consider a commutative square
$$
\xymatrix@=3pc{
K \ar[r]^{f} & L \\
A\ar [u]^{u} \ar [r]_{g} & B \ar[u]_{v}
}
$$
with $A$ and $B$ finite-dimensional and $\eps>0$. Let $v=v_2v_1$ be a (dense, isometry) factorization of $v$. Let 
$$
\xymatrix@=3pc{
K \ar[r]^{f} & L \\
C\ar [u]^{\overline{v}} \ar [r]_{\overline{f}} & B_1 \ar[u]_{v_2}
}
$$
be a pullback. There is $w:A\to C$ such that $\overline{v}w=u$ and
$\overline{f}w=v_1g$. Take $\delta=\frac{\eps}{1-\eps}$. There is a linear operator $t:B_1\to K$ such that $t\overline{f}=\overline{v}$ and 
$\pa t\pa\leq 1+\delta$. Then $t'=\frac{1}{1+\delta}t$ has norm $\leq 1$. For $x\in C$, $\pa x\pa\leq 1$, we have
$$
\pa t'\overline{f}x-\overline{v}x\pa=\pa \frac{t\overline{f}x}{1+\delta}-\overline{v}x\pa=
\pa\frac{\overline{v}x}{1+\delta}-\overline{v}x\pa=
\frac{\delta}{1+\delta}\pa \overline{v}x\pa\leq\frac{\delta}{1+\delta}=\eps.
$$
Hence $t'\overline{f}\sim_\eps\overline{v}$. Therefore
$t'v_1g=t'\overline{f}w\sim_\eps\overline{v}w=u$. We have
proved that $f$ is pure.
 
II. Let $f:K\to L$ be pure. Consider a finite-dimensional subspace $B$ of $L$ and $\eps>0$. Let $v:B\to L$
be the embedding and take the pullback
$$
\xymatrix@=3pc{
K \ar[r]^{f} & L \\
A\ar [u]^{\overline{v}} \ar [r]_{\overline{f}} & B \ar[u]_{v}
}
$$
Then $\overline{f}:K\cap B=A\to B$ is the embedding of a finite-dimensional subspace of $B$. Let $e_1,\dots,e_m,\dots,e_n$
be a basis of $B$ such that $e_1\dots,e_m$ is a basis of $A$.
Since any two norms on a finite-dimensional Banach space are equivalent, there is a number $r\geq 1$ such that 
$$
\sum_{0<i\leq n} |a_i|\leq r\parallel\sum_{0<i\leq n} a_ie_i\parallel.
$$
Let $\delta = \frac{\varepsilon}{r}$. Since $f$ is pure,
there exists $t':B\to K$ such that $t'\overline{f}\sim_\delta\overline{v}$. Let $t:B\to K$ be defined as follows: $t(e_i)=e_i$ for $i=1\dots,m$ and $t(e_i)=t'(e_i)$ for
$i=m+1,\dots n$. Consider $a\in B$, $\pa a\pa\leq 1$. Then
\begin{align*}
\pa t(a)\pa &=\pa\sum_{0<i\leq m}a_ie_i +\sum_{m<i\leq n}a_it'(e_i)\pa
=\pa\sum_{0<i\leq m}a_i(e_i-t'(e_i))+\sum_{0<i\leq n}a_it'(e_i)\pa\\
&\leq \pa \sum_{0<i\leq m} a_i(e_i-t'(e_i))\pa + \pa t'(a)\pa\leq
\pa \sum_{0<i\leq m} a_i(e_i-t'(e_i))\pa + \pa a\pa\\
&\leq \sum_{0<i\leq m}|a_i|\delta + \pa a\pa\leq \sum_{0<i\leq n}|a_i|\delta + \pa a\pa\leq
r\delta\pa\sum_{0<i\leq n}a_ie_i\pa+\pa a\pa\\
&=\eps\pa a\pa +\pa a\pa=(1+\eps)\pa a\pa .
\end{align*}
Hence $\pa t\pa\leq 1+\eps$. Since $t(x)=x$ for $x\in A$, we have proved that $K$ is an ideal in $L$.
\end{proof}

\begin{coro}\label{ideal2}
A morphism in $\Ban$ is pure if and only it it is an embedding of an ideal.
\end{coro}
\begin{proof}
It follows from \ref{pure6} and the fact that an isometry $f:K\to L$ makes $f[K]$ a closed subspace of $L$.
\end{proof}

\begin{propo}\label{stable}
Pure morphisms in $\Ban$ are stable under $\eps$-pushouts for every $\eps>0$.
\end{propo}
\begin{proof}
Consider an $\eps$-pushouts
$$
\xymatrix@=3pc{
\bar{K} \ar[r]^{\bar{f}} & \bar{L} \\
K\ar [u]^{u} \ar [r]_{f} & L \ar[u]_{\bar{u}}
}
$$ 
where $f$ is pure. Following \ref{char}, there are $g,h$
such that $gf\sim_\eps h$, with $h$ being a directed colimit of split monomorphisms. Take the pushout
$$
\xymatrix@=3pc{
\bar{K} \ar[r]^{\bar{h}} & \bar{M} \\
K\ar [u]^{u} \ar [r]_{h} & M \ar[u]_{\bar{\bar{u}}}
}
$$ 
Since
$\bar{\bar{u}}gf\sim_\eps\bar{\bar{u}}h=\bar{h}u$, there is a unique
$\bar{g}:\bar{L}\to\bar{M}$ such that $\bar{g}\bar{f}=\bar{h}$ and $\bar{g}\bar{u}=\bar{\bar{u}}g$. Since split monomorphisms are stable under pushouts and pushouts commute with directed colimits, $\bar{h}$
is a directed colimit of split monomorphisms. Then \ref{char} implies that $\bar{f}$ is pure.
\end{proof}

\begin{propo}\label{preserve}
Let $F:\Ban\to\Ban$ be an enriched functor preserving directed colimits.
Then $F$ preserves pure morphisms.
\end{propo}
\begin{proof}
$F$ clearly preserves directed colimits of split monomorphisms. Since $F$ is enriched $f\sim_\eps g$ implies $Ff\sim_\eps Fg$. Thus the result follows from \ref{char}.
\end{proof}

\begin{rem}\label{preserve1}
{
\em
The projective tensor product functor $K\otimes-:\Ban\to\Ban$ is enriched and preserves all colimits. Hence, by \ref{preserve}, it preserves pure morphisms. Thus it preserves ideals, which was proved in
\cite[Theorem 1]{Ra}. Add that it does not need to preserve isometries (see \cite{Ry}, p.18).
}
\end{rem}

\begin{rem}\label{dual}
{
\em
Following \cite[Lemma 1(i)]{Ra}, the canonical embedding $d_X:X\to X^{\ast\ast}$ is an ideal. Recall that $X^\ast$ is the space $\{X,\Bbb C\}$ of bounded linear mappings $X\to\Bbb C$.
}
\end{rem}

Recall that a Banach space $K$ is \textit{injective} with respect to $h:A\to B$ in $\Ban$ if for every $f:A\to K$ there exists $g:B\to K$ such that $gh=f$.

\begin{defi}\label{pure-inj}
{
\em
Banach spaces injective with respect to pure morphisms will be called \textit{pure-injective}.
}
\end{defi}

\begin{rem}\label{inj}
{
\em
Injective Banach spaces, i.e., injective with respect to isometries, are pure-injective.
}
\end{rem}

\begin{lemma}\label{pure-inj1}
Banach spaces $K^\ast$ are pure-injective.
\end{lemma}
\begin{proof}
Let $h:A\to B$ be pure and $f:A\to K^\ast$. Since $K\otimes-$ is left adjoint to $\{K,-\}$, we get the adjoint transpose $\tilde{f}:K\otimes A\to\Bbb C$ and the pure morphism $K\otimes h:K\otimes A\to K\otimes B$ (see \ref{preserve1}). Since $\Bbb C$ is injective, there is $g:K\otimes B\to\Bbb C$
such that $g(K\otimes h)=\tilde{f}$. Then the adjoint transpose $\tilde{g}:B\to K^\ast$ satisfies $\tilde{g}h=f$. Hence $K^\ast$ is 
pure-injective.
\end{proof}

\begin{rem}\label{wfs}
{
\em
$\Ban$ has enough pure-injectives because every Banach space admits a pure morphism $K\to K^{\ast\ast}$ and $K^{\ast\ast}$ is pure-injective. 

Let $\ci$ denote the class of all ideals and $\ci^\square$ consist
of morphisms $g:C\to B$ such that in every commutative square 
$$
\xymatrix@=3pc{
B \ar[r]^{u} & D \\
A\ar [u]^{f} \ar [r]_{v} & C \ar[u]_{g}
}
$$ 
there exists $t:B\to D$ such that $tf=v$ and $gt=u$.
Following \ref{cancell} and \cite[1.6]{AHRT}, $(\ci,\ci^\square)$ is a weak factorization system. 
 
The category of Banach spaces with ideals as morphisms is $\aleph_1$-accessible. This follows from \ref{closure} and the fact for every separable subspace $A$ of $K$ there is a separable ideal $B$ of $K$ containing $A$ (see \cite{SY}). Analogously to \cite[9.10]{LRV}, this category has the order property and thus it does not have a stable independence notion. Following \cite[3.1]{LRV1}, $\ci$ is not cofibrantly generated. Add that also isometries are not cofibrantly generated (see \cite[3.1(2)]{R}).
}
\end{rem}

\section{Model theory}
Pure morphisms in locally finitely presentable categories have a model-theoretic characterization as morphisms elementary with respect to 
positive-primitive formulas (see \cite[5.15]{AR1}). We will prove an analogical result for Banach spaces. The needed logic is closely related to the logic of positive bounded formulas (see \cite{H} and \cite{I}). 

Let $L$ be the language consisting of a binary operation symbol $+$, a nullary ope\-ration symbol $0$ and unary operation symbols $r\cdot-$
where $r$ is a rational complex number. Moreover, for every rational number
$M$, we have a unary relation symbol $\pa-\pa\leq M$. Terms
are given in a usual way using operation symbols. Atomic formulas are either $t_1=t_2$ where $t_1$ and $t_2$ are terms or $\pa t\pa\leq M$.
The notation $\varphi(x_1,\dots,x_n)$ for an atomic formula means that all variables in $\varphi$ are among $x_1,\dots,x_n$.

\textit{Positive-primitive} formulas are
$$
\varphi(x_1,\dots x_n)=(\exists y_1,\dots y_m)(\bigwedge\limits_{i<\omega}\psi_i(x_1,\dots,x_n,y_1,\dots,y_m))
$$
where $\psi_i$ are atomic formulas. \textit{Approximations} of $\varphi$
are formulas $\varphi_\eps$, $\eps>0$ where every occurence of an atomic formula $\pa t\pa\leq M$ is replaced by $\pa t\pa\leq M+\eps$. 
In particular,
$$
\varphi_\eps(x_1,\dots x_n)=(\exists y_1,\dots y_m)(\bigwedge\limits_{i<\omega}(\psi_i)_\eps(x_1,\dots,x_n,y_1,\dots,y_m)).
$$

For a Banach space $K$ and $a_1,\dots,a_n\in K$, the meaning of
$$
K\models\varphi[a_1,\dots,a_n]
$$
is evident.
We will write 
$$
K\models_{\ap}\varphi[a_1,\dots,a_n]
$$
if 
$$
K\models\varphi_\eps[a_1,\dots,a_n]
$$
for every $\eps>0$. Clearly,
$$
K\models\varphi[a_1,\dots,a_n]\Rightarrow K\models_{\ap}\varphi[a_1,\dots,a_n]
$$
but not conversely.
 
\begin{rem}\label{findim2}
{
\em
Let $A$ be a finite-dimensional Banach space with a basis $e_1,\dots,e_n$.
Consider the the conjunction of atomic formulas
$$
\pi_A(x_1,\dots,x_n)=\bigwedge\pa \sum\limits_{i=1}^n r_ix_i\pa \leq\pa \sum\limits_{i=1}^n r_ie_i\pa
$$
where the conjunction is indexed by all $n$-tuples $(r_1,\dots,r_n)$ of rational complex numbers. Clearly,
$$
K\models\pi_A[a_1,\dots,a_n]
$$
for $a_1\dots,a_n\in K$ if and only if the linear map $f:A\to K$ such that $f(e_i)=a_i$ for $i=1,\dots,n$ has norm $\leq 1$. In analogy to
\cite[5.5]{AR1} we call $\pi_A$ the \textit{presentation formula} of $A$. Similarly,
$$
K\models(\pi_A)_\eps[a_1,\dots,a_n]
$$ 
if the linear map $f:A\to K$ such that $f(e_i)=a_i$ for $i=1,\dots,n$ has norm $\leq 1+\eps$. 
}
\end{rem}

\begin{propo}\label{ideal3} 
A closed subspace $K$ of a Banach space $L$ is an ideal if and only if for each positive-primitive formula $\varphi(x_1,\dots,x_n)$ and each assignment of values $a_1,\dots,a_n\in K$ we have
$$
K\models_{\ap}\varphi[a_1,\dots,a_n] \Leftrightarrow 
L\models_{\ap}\varphi[a_1,\dots,a_n].
$$
\end{propo}
\begin{proof}
Let $K$ be a closed subspace of $L$, $\varphi(x_1,\dots,x_n)$ a  positive-primitive bounded formula and $a_1,\dots,a_n\in K$. Obviously,
$$
K\models_{\ap}\varphi[a_1,\dots,a_n] \Rightarrow 
L\models_{\ap}\varphi[a_1,\dots,a_n].
$$ 
Thus, it is only the reverse implication that matters. 

I. Let $K$ satisfy the reverse implication and consider a finite-dimensional subspace $B$ of $L$ and $\eps>0$. Let $e_1,\dots,e_n$
be a basis of $A=B\cap K$ and $e_1,\dots,e_n,\dots ,e_{n+m}$ a basis of $B$.
Consider the formula 
$$
\varphi(x_1,\dots.x_n)=(\exists x_{n+1},\dots,x_{n+m})\pi_B(x_1,\dots,x_{n+m})
$$ 
where $\pi_B$ is the presentation formula of $B$.
Since 
$L\models\varphi[e_1,\dots,e_n]$, we have $$
K\models_{\ap}\varphi[e_1,\dots,e_n].
$$ 
Following \ref{findim2}, there exists a linear map $f:B\to K$ of norm $\leq 1+\eps$ such that $f(e_i)=e_i$ for $i=1,\dots,n$. Hence $K$ is an ideal in $L$.

II. For the converse, let $K$ be an ideal in $L$ and consider 
$a_1,\dots,a_n\in K$ and a positive-primitive formula 
$$
\varphi(x_1,\dots,x_n)=(\exists y_1,\dots,y_m)\psi(x_1\dots,x_n,y_1,\dots,y_m)
$$
such that $L\models_{\ap}\varphi[a_1,\dots,a_n]$. Let 
$$
\psi(x_1,\dots,x_n,y_1,\dots,y_m)=\bigwedge\limits_{i<\omega}
\psi_i(x_1,\dots,x_n,y_1,\dots,y_m)
$$
where $\psi_i(x_1,\dots,x_n,y_1,\dots,y_m)$ are atomic formulas. Thus
there are $b_1,\dots,b_m\in L$ such that
$L\models\psi_{\eps_1}[a_1,\dots,a_n,b_1,\dots,b_m]$ for every $\eps_1>0$.
Let $B$ be the subspace of $L$ generated by $a_1,\dots,a_n,b_1,\dots,b_m$.

Since $K$ is an ideal in $L$, for every $\eps_2>0$, there is a linear map $f:B\to K$ such that $\pa f\pa\leq 1+\eps_2$ and $f(a_i)=a_i$ for every $j=1,\dots,n$. Now, let $\eps>0$ and take $\eps_1=\eps_2=\frac{\eps}{2}$.
For $\psi_i=\pa t\pa\leq M$, we have
$$
\pa t[a_1,a_n,\dots,f(b_1)\dots,f(b_m)]\pa\leq
\pa t[a_1,a_n,\dots,b_1\dots,b_m]\pa +\frac{\eps}{2}\leq M+\frac{\eps}{2}
+\frac{\eps}{2}=M+\eps.
$$
Hence $K\models_{\ap}\varphi[a_1,\dots,a_n]$.
\end{proof}

A morphism $f:K\to L$ in $\Ban$ is a \textit{u-extension} if there exists an ultrapower $K^\cu$ and an isometry $g:L\to K^\cu$ such that $gf=d_K$
where $d_K:K\to K^\cu$ is the canonical embedding of $K$ to its ultrapower (see \cite[6.3]{S}). Concerning ultraproducts of Banach spaces
one can also consult \cite[p. 7]{I} or \cite[14.1]{G}). Let $K$ be a closed subspace of a Banach space $L$. Following \cite[6.7]{S} (see also \cite[14.2.7]{G}), the embedding $f:K\to L$ is a \textit{u-extension} if and only if for every finite-dimensional subspace $B$ of $L$ and every $\eps>0$,  there exists a strong $\eps$-isometry $t:B\to K$ such that $tx=x$ for every $x\in K\cap B$. Hence u-extensions are ideals. In \cite[1.3]{ALN}, u-extensions are called \textit{almost isometric ideals}. 

\begin{rem}\label{u-ext}
{
\em
(1) Following \ref{equivalent}(3), $f:K\to L$ is an u-extension if and only if for every finite-dimensional subspace $B$ of $L$ and every $\eps>0$,  there exists an $\eps$-isometry $t:B\to K$ such that $tx=x$ for every $x\in K\cap B$. 

(2) Let $A$ be a finite-dimensional Banach space with a basis $e_1,\dots,e_n$. Following \cite[2.1]{GK1}, let $M=\max |a_i|$ where $a=\sum a_ie_i$ for $\pa a\pa\leq 1$. Let $\delta=\frac{\eps}{nM}$.
Then for every Banach space $K$, for every linear map $f:A\to X$ the following implication holds
$$
\max\limits_{i\leq n}\pa f(e_i)\pa\leq\delta\Rightarrow \pa f\pa\leq\eps.
$$
}
\end{rem}
 
\begin{lemma}\label{u-ext1}
Let $f:K\to L$ be an isometry in $\Ban$ such that for every  $\eps>0$ and every commutative square of isometries
$$
\xymatrix@=3pc{
K \ar[r]^{f} & L \\
A\ar [u]^{u} \ar [r]_{g} & B \ar[u]_{v}
}
$$
with $A$ and $B$ finite-dimensional, there exists an isometry $t:B\to K$ such that $tg\sim_\eps u$. Then $f$ is a u-extension. 
\end{lemma}
\begin{proof}
Take $\eps>0$. Let $e_1,\dots,e_m,\dots,e_n$ be a basis of $B$ such that $e_1,\dots,e_m$ is a basis of $A$. Choose $\delta$ from \ref{u-ext}(2) for the basis $e_1,\dots,e_n$ and $\eps$. Let $t':B\to K$ be an isometry
such that $t'g\sim_\delta u$. Take $t:B\to K$ such that $t(e_i)=e_i$ for
$i=1,\dots,m$ and $t(e_i)=t'(e_i)$ for $i=m+1,\dots,n$. Since 
$\pa (t'-t')(e_i)\pa\leq\delta$ for every $i=1,\dots,n$,
\ref{u-ext}(2) implies that $\pa t-t'\pa\leq\eps$. Hence
$$
|\pa t(x)\pa-\pa x\pa|=|\pa t(x)-\pa t'(x)\pa|\leq\pa t(x)-t'(x)\pa\leq\eps
$$
for every $x$ such that $\pa x\pa\leq 1$. Following \ref{equivalent}(1),
$t$ is an $\eps$-isometry and the result follows from \ref{u-ext}(1).
\end{proof}

\begin{rem}\label{u-ext2}
{
\em
(1) Concerning a converse, in I. of the proof of \ref{ideal}, $t$ is injective (see \ref{equivalent}4). But we would need that it is an isometry.  

(2) The property in \ref{u-ext1} says that $f$ is pure in the category
$\Ban_{\iso}$ of Banach spaces with isometries as morphisms.
}
\end{rem}

\begin{propo} \label{ultra}
Let $\cu$ be an ultrafilter on a set $I$, $K$ a Banach space
and $d_K:K\to K^\cu$ the canonical embedding of $K$ to its ultrapower. Then
$d_K$ is pure.
\end{propo}
\begin{proof}
The ultrapower $K^\cu$ is a directed colimit of powers $K^U$, $U\in\cu$
(cf. \cite[5.17]{AR1}) and $d_K$ is a directed colimit of diagonals $d_U:K\to K^U$. Since $d_U$ are split monomorphisms, 
\ref{char} implies that $d_K$ is pure.
\end{proof}

\begin{lemma}\label{ultra2}
Ultrapowers preserve pure morphisms.
\end{lemma}
\begin{proof}
Let $f:K\to L$ be pure. We have to show that $f^\cu:K^\cu\to L^\cu$ is pure. It is easy to see that $f^U:K^U\to L^U$ is pure for every $U\in\cu$.
Indeed, consider a commutative square 
$$
\xymatrix@=3pc{
K ^U\ar[r]^{f^U} & L^U \\
A\ar [u]^{u} \ar [r]_{g} & B \ar[u]_{v}
}
$$
with $A$ and $B$ finite-dimensional. Let $p^i:K^U\to K$ and $q_i:L^U\to L$
be projections for $i\in U$. Let $\eps>0$. Then, for every $i\in U$, there
is $t_i:B\to K$ such that $t_ig\sim_\eps p_iu$. Then the induced morphism 
$t:B\to K^U$ satisfies $tg\sim_\eps u$.

Since $f^\cu$ is a directed colimit of $f^U$, the result follows 
from \ref{closure}.
\end{proof}

\begin{rem}\label{ultra3}
{
\em
The category $\Mod(R)$ of modules over a ring $R$ is locally finitely presentable. Pure-injective modules (i.e., modules injective with respect to pure morphisms) coincide with algebraically compact ones (see \cite[4.3.11]{Pr}). There is a set $I$ and an ultrafilter $\cu$ on $I$ such that $K^\cu$ is algebraically compact for every module $K$ 
(\cite[4.2.18]{Pr}). Hence pure morphisms and u-extensions coincide. 

Indeed, let $f:K\to L$ be pure. Like in \ref{ultra2}, $f^\cu$ is pure. Since $K^\cu$ is pure-injective, $f^\cu$ is split by $s$. Then $sd_L$ makes $f$ an u-extension. Conversely, in the same way as in \ref{ultra},
one proves that the canonical embedding $d_K:K\to K^\cu$ is pure. Since
pure morphisms are left-cancellable, every u-extension is pure.
}
\end{rem}

\section{Injectivity}
\begin{defi}\label{apinj}
{
\em
Given a class $\ch$ of morphisms in a $\CMet$-enriched category $\ck$, an object $K$ is called \textit{approximately injective} with respect to $\ch$ if for every $h:A\to B$ in $\ch$, every $\eps>0$ and every morphism $f:A\to K$ there exists a morphism $g:B\to K$ such that $gh\sim_\eps f$.

The class of objects approximately injective with respect to $\ch$ will be denoted as $\Inj_{\ap}\ch$.
}
\end{defi}

\begin{propo}\label{adj2}
Let $U:\ck\to\cl$ and $F:\cl\to\ck$ be an enriched adjunction. Then $UK$ is approximatively injective with respect to $\ch$ iff $K$ is approximately injective with respect to $F(\ch)$.
\end{propo}
\begin{proof}
$UK$ is approximately injective with respect to $\ch$ iff for every $h:A\to B$ in $\ch$ and every morphism $f:A\to UK$ there exists a morphism $g:B\to UK$ such that $gh\sim_\eps f$.
This is the same as $\tilde{g}Fh= \tilde{gh}\sim_\eps \tilde{f}$, which is equivalent to $K$ being approximately injective with respect to $F(\ch)$. 
\end{proof}

\begin{rem}\label{lind}
{
\em
Lindenstrauss spaces are Banach spaces which are ideals in every superspace
(\cite[4.1]{ALN}). Hence $K$ is a Lindenstrauss space if and only if every isometry $K\to L$ is pure. This implies the well-known characterization of Lindenstrauss spaces as Banach spaces approximately injective with respect to isometries between finite-dimensional Banach spaces.

Indeed, let $K$ be Lindenstrauss, $h:A\to B$ be an isometry between finite-dimensional Banach spaces and $f:A\to K$. Consider the pushout
$$
\xymatrix@=3pc{
K \ar[r]^{\bar{h}} & \bar{L} \\
A\ar [u]^{f} \ar [r]_{h} & B \ar[u]_{\bar{f}}
}
$$ 
Since $\bar{h}$ is an isometry, it is pure and thus, for every $\eps>0$,  there is $g:B\to K$ such that $gh\sim_\eps f$. Hence $L$ is
approximately injective with respect to $h$. The converse is evident.
}
\end{rem} 

\begin{lemma}\label{def}
Let $\ch$ be a set of morphisms between finite-dimensional Banach spaces.
Then $\Inj_{\ap}\ch$ is closed under products, directed colimits and ideals.
\end{lemma}
\begin{proof}
Following \cite[4.3]{RT}, $\Inj_{\ap}\ch$ is closed under products. Let $(k_i:K_i\to K)_{i\in I}$ be a directed colimit of $K_i\in\Inj_{\ap}\ch$. Consider $h:A\to B$ in $\ch$, $f:A\to K$ and take $\eps>0$. Following \ref{finpres}, there exists $i\in I$ and $f':A\to K_i$ such that $k_if'\sim_{\frac{\eps}{2}} f$. There
is $g:B\to K_i$ such that $gh\sim_{\frac{\eps}{2}} f'$. Hence 
$$
k_igh\sim_{\frac{\eps}{2}}k_if'\sim_{\frac{\eps}{2}} f
$$
Thus $(k_ig)h\sim_\eps f$. Hence $K\in\Inj_{\ap}$.

The closure if $\ck$ under ideals follows from \cite[4.5]{RT}, \cite[3.6(2)]{RT} and \ref{ideal}.
\end{proof}

\begin{rem}\label{def1}
{
\em
(1) In locally finitely presentable categories, classes of objects closed under products, directed colimits and pure submodules are called \textit{definable} (see \cite{KR}). They coincide with classes of objects injective with respect to a set of morphisms between finitely presentable objects (see \cite[4.11]{AR1}). We do not know whether this holds in $\Ban$.

(2) Let $\lambda$ be an uncountable regular cardinal. In \cite[5.5]{RT}, there are characterized classes $\Inj_{\ap}\ch$ of Banach spaces where the domains and the codomains of morphisms in $\ch$ have density character $<\lambda$. They are classes closed under products,
$\lambda$-directed colimits and $\lambda$-pure subspaces. Here, $\lambda$-pure morphisms are defined as in \ref{pure} with $A$ and $B$ of density character $<\lambda$. Following \ref{weakly}, they coincide with weakly
$\lambda$-ap-pure morphisms.
}
\end{rem}

\begin{defi}\label{ideal4}
{
\em
A morphism $f:K\to L$ in $\Ban$ will be called $(\omega,\omega_1)$-pure if in every commutative square
$$
\xymatrix@=3pc{
K \ar[r]^{f} & L \\
A\ar [u]^{u} \ar [r]_{g} & B \ar[u]_{v}
}
$$
where $A$ is finite-dimensional and $B$ separable and for every $\eps>0$ there exists $t:B\to K$ such that $tg\sim_\eps u$.
}
\end{defi}

\begin{theo}\label{char1}
Classes of Banach spaces closed under products, directed colimits and $(\omega,\omega_1)$-pure subspaces are precisely classes of Banach spaces approximately injective with respect to a set of morphisms having finite-dimensional domain and separable codomain.
\end{theo}
\begin{proof}
I. Like in \ref{def}, $\Inj_{\ap}\ch$ is closed under products and directed colimits. The closure under $(\omega,\omega_1)$-pure subspaces
follows from the proof of \cite[4.5]{RT}. Indeed, let $f:K\to L$ be 
$(\omega,\omega_1)$-pure and $L\in\Inj_{\ap}\ch$. Consider $u:A\to K$, $h:A\to B$ in $\ch$ and $\eps>0$. There exists $v:B\to L$ such that $vh\sim_{\frac{\eps}{2}}fu$. Following \ref{weakly}, there is $t:B\to K$ such that $th\sim_\eps u$. Hence $K\in\Inj_{\ap}\ch$. 

II. Let $\cl$ be closed under products, directed colimits and $(\omega,\omega_1)$-pure subspaces. Like in the proof of \cite[4.8]{RT}, $\cl$ is weakly reflective in $\Ban$. Hence every Banach space $K$ comes with a morphism $r_K:K\to K^\ast$, $K^\ast\in\cl$ such that every object of $\cl$ is injective with respect to $r_K$. Let $\ch$ consist of all morphisms $f:A\to B$ such that $A$ is finite-dimensional, $B$ is separable and every object of $\cl$ is approximately injective with respect to $f$. We have to prove that $\Inj_{\ap}\ch\subseteq\cl$. For this, it suffices to show that, for $K\in\Inj_{\ap}\ch$,  any weak reflection $r:K\to K^\ast$ is $(\omega,\omega_1)$-pure.

Thus, given $K\in\Inj_{\ap}\ch$ and a weak reflection $r:K\to K^*$ in $\cl$, we are to prove that in any commutative square
\vskip 2 mm
$$
\xymatrix@C=3pc@R=3pc{
A \ar[r]^{h} \ar [d]_u &
B \ar [d]^{v}\\
K\ar[r]_{r} & K^*
}
$$
\vskip 2 mm\noindent
with $A$ finite-dimensional and $B$ separable, the morphism $u$ $\eps$-factors through $h$.  

\vskip 2mm
\noindent {\it Claim}: For every $\eps>0$, there is an $\eps$-factorization $u\sim_{\eps} u_2\cdot u_1$ and an $\eps$-homomorphism $(u_1,v_1):h\to \bar r$ 
(i.e., $v_1h\sim_\eps \bar{r}u_1$) where $\bar r:\bar K\to\bar K^*$ 
is a weak reflection into $\cl$ of a finite-dimensional $\bar K$.

\vskip2mm
\noindent {\it Proof of claim}. Following the proof of \cite[4.8]{RT}, for every $\eps>0$ there is a factorization $u=u_2\cdot u_1$ and an $\eps$-homomorphism $(u_1,v_1):h\to \bar r$ where 
$\bar r:\bar K\to\bar K^*$ is a weak reflection into $\cl$ of a separable $\bar K$. Assume that the claim is false for $\eps>0$. Then it is false for $\frac{\eps}{6}$. Take an $\frac{\eps}{6}$-homomorphism $(u_1,v_1)$ from the claim with $\bar{K}$ separable.
We express $\bar K$ as a colimit of a smooth chain $k_{ij}:K_i\to K_j$ $(i\leq j<\omega)$ of finite-dimensional subspaces $K_i$. This provides weak reflections $r_i:K_i\to K_i^\ast$ into $\cl$ such that their colimit $r_\omega:\bar K\to K_\omega^\ast$ factorizes through $\bar r$, i.e.,
$r_\omega=s\bar r$ for some $s:\bar K^\ast\to K_\omega^\ast$. Since 
$\bar ru_1\sim_{\frac{\eps}{6}} v_1h$, we have 
$r_\omega u_1=s\bar ru_1\sim_{\frac{\eps}{6}} sv_1h$, so that $(u_1,sv_1):h\to r_\omega$ is an $\frac{\eps}{6}$-homomorphism. 
Following \ref{finpres}, there is $n<\omega$, $u':A\to K_n$ and 
$v':B\to K^\ast_n$ such that $k_nu'\sim_{\frac{\eps}{6}}u_1$ and 
$k_n^\ast v'\sim_{\frac{\eps}{6}}sv_1$. Hence
$$
k^\ast_nv'h\sim_{\frac{\eps}{6}}sv_1h\sim_{\frac{\eps}{6}}r_\omega u_1
\sim_{\frac{\eps}{6}}r_\omega k_nu'=k_n^\ast r_nu'.
$$
Hence $k_n^\ast v'h\sim_{\frac{\eps}{2}} k_n^\ast r_nu'$. Following
\ref{finpres}, there is $n\leq m<\omega$ such that $$
k^\ast_{nm}v'h\sim_{\eps} k^\ast_{nm}r_nu'=r_mk_{nm}u'.
$$ 
Since 
$$
u_1\sim_{\frac{\eps}{6}}k_nu'=k_mk_{nm}u',
$$ 
we get the claim for $\eps$, which yields a contradiction. 

We are ready to prove that $u$ $\eps$-factors through $h$. Let us consider an $\frac{\eps}{3}$-factorization and an $\frac{\eps}{3}$-homomorphism $(u_1,v_1):h\to \bar r$ as in the above claim for. Let us express $\bar K^*$ as a $\omega_1$-directed colimit of separable Banach spaces $K_i$, $i\in I$, with a colimit cocone $k_i:K_i\to\bar K^*$. Since $\bar K$ is finite-dimensional and $B$ separable, there is $i\in I$, $\tilde r$ and $\bar{v}_1$ such that $k_i\bar{v}_1=v_1$ and $k_i\tilde{r}=\bar{r}$  
\vskip 2 mm
$$
\xymatrix@C=3pc@R=3pc{
A \ar [dd]_{u_1}
  \ar [rr]^{h} &&
B \ar[dd]^{v_1}
  \ar [dl]_{\tilde v_1}  \\
& K_i\ar[dr]^{k_i}  &\\
\bar K \ar[ur]^{\tilde r}
  \ar[rr]_{\bar r}  && \bar K^*
}
$$
\vskip 2 mm\noindent
Since all objects of $\cl$ are injective with respect to $\bar r$, $\tilde r\in\ch$. Thus $K$ is approximate injective with respect to $\tilde r$. Choosing $t:K_i\to K$ with 
$u_2\sim_{\frac{\eps}{3}} t\tilde r$ we obtain
$$
u\sim_{\frac{\eps}{3}}u_2u_1\sim_{\frac{\eps}{3}} t\tilde ru_1\sim_{\frac{\eps}{3}} t\tilde v_1h.
$$
Thus $u\sim_\eps t\bar{v}_1h$. Hence, $r$ is $(\omega,\omega_1)$-pure.
\end{proof}

\begin{rem}
{
\em
We do not know whether classes of Banach spaces closed under products, directed colimits and pure subspaces are precisely classes of Banach spaces approximately injective with respect to a set of morphisms having finite-dimensional domains and codomains.
}
\end{rem}

\section{Saturation}
\begin{defi}\label{apsat}
{
\em
Given a class $\ch$ of isometries in a $\CMet$-enriched category $\ck$, an object $K$ is called \textit{approximately saturated} with respect to $\ch$ if for every $h:A\to B$ in $\ch$, every $\eps>0$ and every isometry $f:A\to K$ there exists an isometry $g:B\to K$ such that $gh\sim_\eps f$.

The class of objects approximately saturated with respect to $\ch$ will be denoted
 as $\Sat_{\ap}\ch$.
}
\end{defi}

\begin{lemma}\label{apsat1}
If $\ck$ has conical colimits, $\ck_0$ is wellpowered and $\ch$ is stable under pushouts then every objects approximately saturated with respect to $\ch$ is approximately injective with respect to $\ch$.
\end{lemma}
\begin{proof}
Consider $h:A\to B$ in $\ch$, $\eps>0$ and $f:A\to K$. Following 
\cite[3.14]{AR}, $f=f_2f_1$ where $f_2$ is an isometry. Let 
$$
\xymatrix@=3pc{
A\ar[r]^{h} \ar[d]_{f_1} & B\ar[d]^{\bar{f}_1} \\
A_1 \ar[r]_{\bar{h}} & B_1
}
$$
be a pushout. Since $\bar{h}\in\ch$, there is an isometry $g_2:B_1\to K$ such that
$g_1\bar{h}\sim_\eps f_2$. Then $g_2\bar{f}_1h\sim_\eps f$.
\end{proof}

\begin{rem}\label{gur} 
{
\em
(1) In order to get \ref{adj2} for approximately saturated objects, we would need that
the adjoint isomorphism $\cl(h,UK)\cong\ck(Fh,K)$ preserves and reflects isometries.

(2) If $\ch$ consists of isometries between finite-dimensional Banach spaces then a Banach space $K$ is approximately saturated with respect to $\ch$ iff it is of almost universal disposition for $\ch$ in the sense 
of \cite[3.1]{ASCGM} (see \cite[(H)]{K}). These spaces are called Gurarii spaces.

(3) Like in \ref{lind}, $K$ is Gurarii if and only if every isometry
$K\to L$ satisfies the condition from \ref{u-ext1}. Following \cite[4.2]{ALN}, $K$ is Gurarii if and only if every isometry $K\to L$ is a u-extension.

(4) \ref{apsat1} yields a well-known fact that every Gurarii space is a Lindenstrauss space (see \cite[4.1 and 4.3]{ALN}).
}
\end{rem} 

\begin{rem}\label{asplit}
{
\em
In \cite[5.23(1)]{AR}, a morphism $f:K\to L$ was called approximately split
if for every $\eps>0$ there is $s:L\to K$ such that $sf\sim_\eps f$.
Since an approximately split morphism is pure, it is an isometry.
}
\end{rem}
 
\begin{lemma}\label{apsat2}
A Banach space $K$ is Lindenstrauss if and only if it has an approximately split morphism $K\to L$ to a Gurarii space $L$. 
\end{lemma}
\begin{proof}
Following \ref{gur}(4), every Gurarii space is Lindenstrauss. Since
approximately split morphisms are pure (see \ref{asplit}), \ref{def} implies that approximately split subspaces of a Gurarii space are Lindenstrauss.

Conversely, let $K$ be a Lindenstrauss space. Following \cite[3.6]{GK0},
there is an isometry $f:K\to L$ where $L$ is Gurarii. Then \ref{lind} and
\ref{ideal} imply that $f$ approximately splits.
\end{proof}

\end{document}